\documentclass[reqno]{amsart}
\usepackage[utf8]{inputenc}
\usepackage[rgb]{xcolor}
\usepackage[colorlinks, allcolors=blue]{hyperref}
\usepackage{amsmath,amscd,amsthm,amssymb,graphicx,array,enumerate,placeins}

\title[Flexibility of entropy of boundary maps]{Flexibility of measure-theoretic entropy of boundary maps associated to Fuchsian groups}

\author{Adam Abrams}
\address{Institute of Mathematics, Polish Academy of Sciences, Warsaw, Poland 00656}
\email{the.adam.abrams@gmail.com}
\author{Svetlana Katok}
\address{Department of Mathematics, The Pennsylvania State University, University Park, PA 16802}
\email{sxk37@psu.edu}
\author{Ilie Ugarcovici}
\address{Department of Mathematical Sciences, DePaul University, Chicago, IL 60614}
\email{iugarcov@depaul.edu}

\dedicatory{In memory of Tolya}
\thanks{The second author was partially supported by NSF grant DMS 1602409. The third author was partially supported by a Simons Foundation Collaboration Grant.}


\allowdisplaybreaks 
\makeatletter \g@addto@macro\@floatboxreset\centering \makeatother 
\providecommand\captionperiod{.} 

\usepackage{cleveref} 
\crefname{subsection}{section}{sections}
\Crefname{subsection}{Section}{Sections}

\newtheorem{thm}{Theorem}
\newtheorem{cor}[thm]{Corollary}
\newtheorem{lem}[thm]{Lemma}
\newtheorem{prop}[thm]{Proposition}
\theoremstyle{definition}

\newtheoremstyle{iremark}{\topsep}{\topsep}{\upshape}{0pt}{\itshape}{.}{5pt plus 1pt minus 1pt}{\thmname{#1}\thmnumber{ \itshape#2}\thmnote{ (#3)}}
\theoremstyle{iremark}
\newtheorem{remark}[thm]{Remark}
\newcommand\<{\begin{equation}} \renewcommand\>{\end{equation}}

\newcommand{\C}{{\mathbb C}}
\newcommand{\G}{\Gamma} 
\newcommand{\D}{{\mathbb D}}

\newcommand{\R}{{\mathbb R}}

\newcommand{\Sb}{{\mathbb S}}
\newcommand{\Bc}{{\mathcal B}}

\newcommand{\Fc}{{\mathcal F}}
\newcommand{\Gc}{{\mathcal G}}

\newcommand{\Pc}{{\mathcal P}}
\newcommand{\Tc}{{\mathcal T}}

\newcommand\abs[1]{\left\vert#1\right\vert}

\newcommand\setbuilder[3][:]{\left\{\, #2 #1 #3 \,\right\}}
\let\tilde\widetilde
\renewcommand\bar[1]{\,\overline{\!#1\!}\,}
\renewcommand\d{\mathrm{d}}
\providecommand\Id{\mathrm{Id}}

\newcounter{commentcounter} \colorlet{gold}{yellow!67!black}

\renewcommand\P{{\bar P}}
\newcommand\geo{_\mathrm{geo}}
\providecommand\reg{}\renewcommand\reg{^\mathrm{reg}}

\begin{document}
\maketitle
\begin{abstract}
Given a closed, orientable, compact surface $S$ of constant negative curvature and genus $g \ge 2$, we study the measure-theoretic entropy of the Bowen--Series boundary map with respect to its smooth invariant measure. We obtain an explicit formula for the entropy that only depends on the perimeter of the $(8g-4)$-sided fundamental polygon of the surface $S$ and its genus. Using this, we analyze how the entropy changes in the Teichm\"uller space of $S$ and prove the following flexibility result: the measure-theoretic entropy takes all values between $0$ and a maximum that is achieved on the surface that admits a regular $(8g-4)$-sided fundamental polygon. We also compare the measure-theoretic entropy to the topological entropy of these maps and show that the smooth invariant measure is not a measure of maximal entropy.
\end{abstract}

\section{Introduction} \label{sec intro}

Any closed, orientable, compact surface $S$ of genus $g \ge 2$ and constant negative curvature can be modeled as $S = \G\backslash\D$, where $\D = \setbuilder{ z \in \C }{ \abs z < 1 }$ is the unit disk endowed with hyperbolic metric
\begin{equation}
    \label{hypmetric} \frac{2\abs{\d z}}{1-{\abs z}^2}
\end{equation}
and $\G$ is a finitely generated Fuchsian group of the first kind acting freely on $\D$.

Recall that geodesics in this model are half-circles or diameters orthogonal to $\Sb=\partial\D$, the circle at infinity. The geodesic flow $\tilde\varphi^t$ on $\D$ is defined as an $\R$-action on the unit tangent bundle $T^1\D$ that moves a tangent vector along the geodesic defined by this vector with unit speed. The geodesic flow $\tilde\varphi^t$ on $\D$ descends to the geodesic flow $\varphi^t$ on the factor $S=\G\backslash\D$ via the canonical projection
\begin{equation*}
    \label{projection} \pi: T^1\D \to T^1S
\end{equation*}
of the unit tangent bundles.  The orbits of the geodesic flow $\varphi^t$ are oriented geodesics on $S$.

A surface $S$ of genus $g$ admits an $(8g-4)$-sided fundamental polygon $\Fc$ obtained by cutting it with $2g$ closed geodesics that intersect in pairs ($g$ of them go around the ``holes'' and another $g$ go around the ``waists'' of $S$) (see \Cref{fig closed geodesics AF}).
\begin{figure}[htb]
    \includegraphics[width=0.8\textwidth]{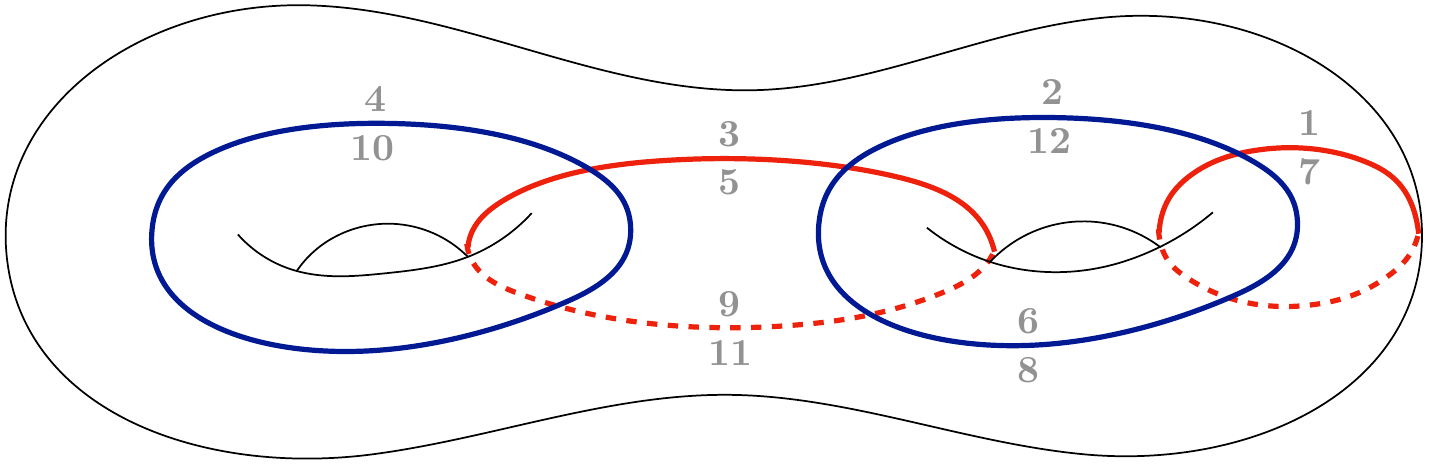}
    \caption{Chain of $2g$ geodesics when $g=2$\captionperiod}
    \label{fig closed geodesics AF}
\end{figure}

The existence of such a fundamental polygon $\Fc$ is an old result attributed \cite{BiS87,W92} to Dehn, Fenchel, Nielsen, and Koebe. Adler and Flatto \cite[Appendix A]{AF91} give a careful proof of existence and properties of $\Fc$.

We label the sides of $\Fc$ in a counterclockwise order by numbers $1 \le i \le 8g-4$ and label the vertices of $\Fc$ by $V_i$ so that side $i$ connects $V_i$ to $V_{i+1}~(\rm{mod}~{8g-4})$ (this gives us a {\em marking} of the polygon). 

We denote by $P_i$ and $Q_{i+1}$ the endpoints of the oriented infinite geodesic that extends side $i$ to the circle at infinity~$\Sb$. (The points $P_i$, $Q_i$ in this paper and~\cite{BS79,KU17,AK19,AKUpre-print} are denoted by $a_i$, $b_{i-1}$, respectively, in~\cite{AF91}.) The order of endpoints on $\Sb$ is the following:
\[ P_1, Q_1, P_2, Q_2, \ldots, P_{8g-4}, Q_{8g-4}.
\label{eq: pq-partition} \]
The identification of the sides of $\Fc$ is given by the side pairing rule
\< \label{sigma} \sigma(i) := \left\{ \begin{array}{ll}
    4g-i \bmod (8g-4) & \text{ if $i$ is odd} \\
    2-i \bmod (8g-4) & \text{ if $i$ is even}.
\end{array} \right. \>
Let $T_i$ denote the M\"obius transformation pairing side $i$ with side $\sigma (i)$. 

\begin{figure}[htb]
	\includegraphics[width=0.5\textwidth]{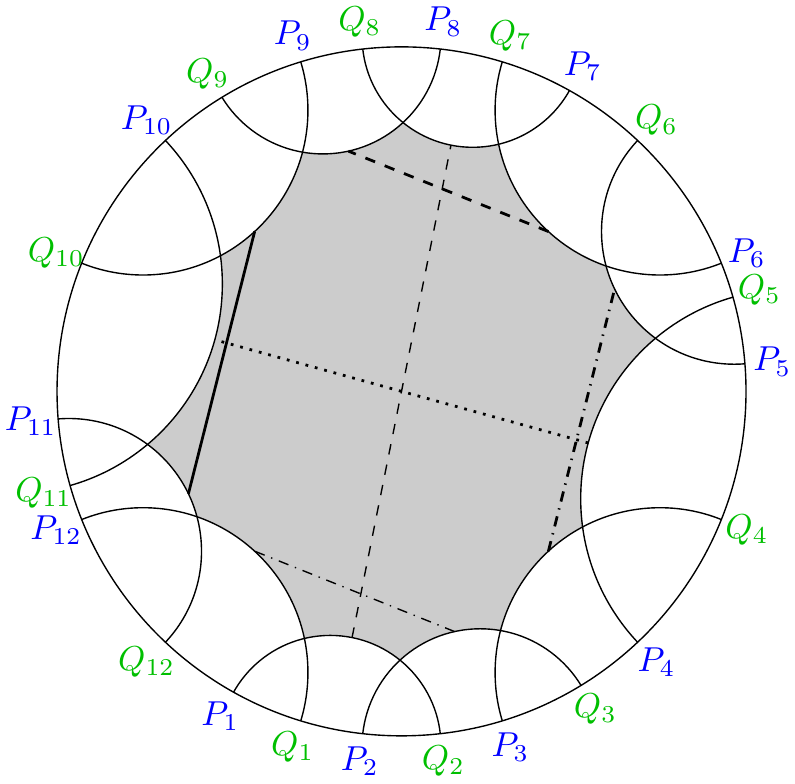} \quad
	\includegraphics[width=0.435\textwidth]{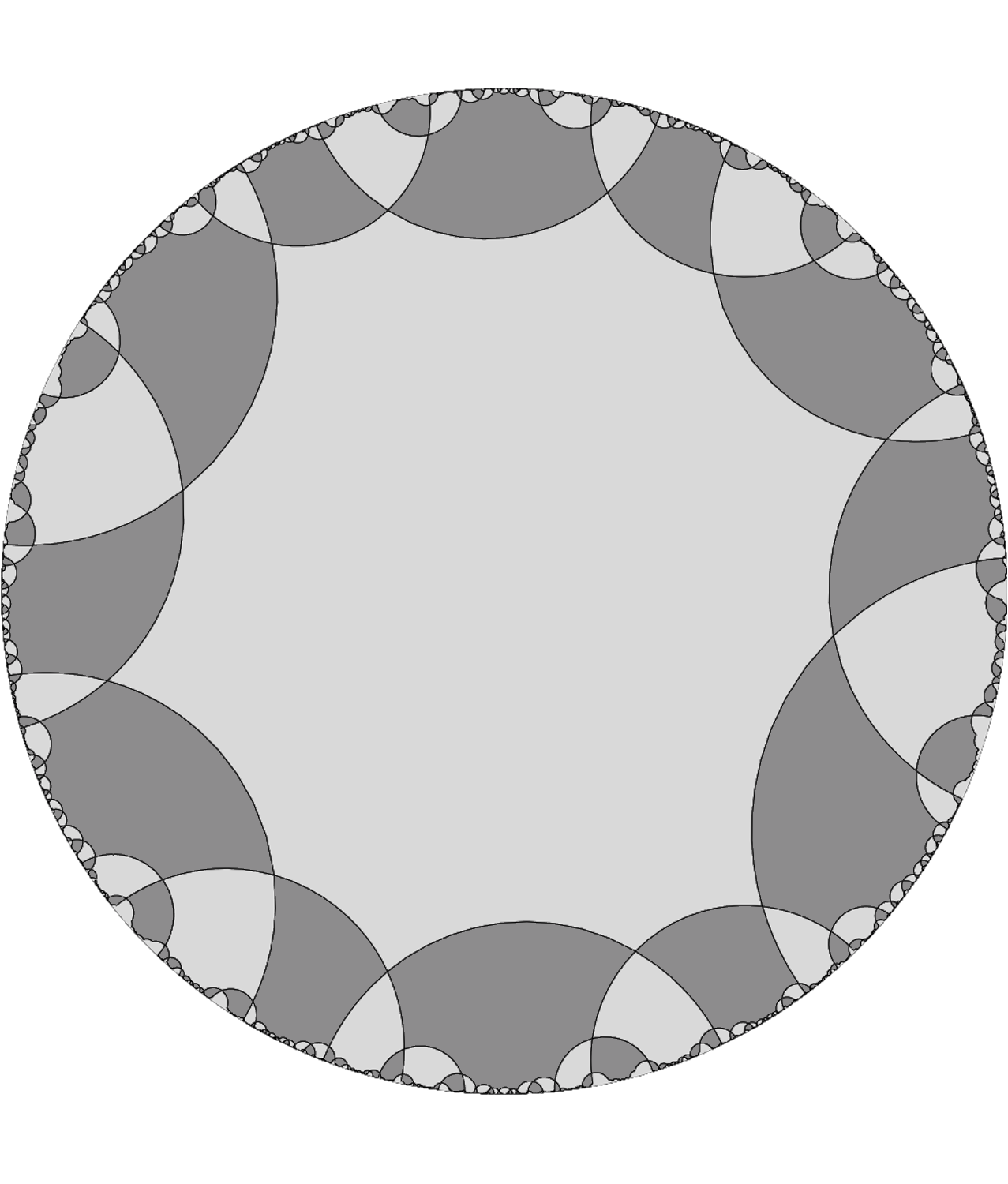}
    \caption{An irregular polygon with side identifications (left) and tessellation (right), genus $2$\captionperiod}
    \label{fig irregular sides}
\end{figure}

Notice that in general the polygon $\Fc$, whose sides are geodesic segments, need not be regular, but the sides $i$ and $\sigma(i)$ must have equal length and the angles at vertices~$i$ and~$\sigma(i)+1$ must add up to $\pi$. The last property implies the ``extension condition,'' which is crucial for our analysis: the extensions of the sides of $\Fc$ do not intersect the interior of the tessellation $\gamma\Fc,\,\gamma\in \G$ (see \Cref{fig irregular sides}). If $\Fc$ is regular (see~\cite[Fig.~1]{AF91}), it is the Ford fundamental domain, i.e., $P_iQ_{i+1}$ is the isometric circle for $T_i$, and $T_i(P_iQ_{i+1})=Q_{\sigma(i)+1}P_{\sigma(i)}$ is the isometric circle for $T_{\sigma(i)}$, so that the inside of the former isometric circle is mapped to the outside of the latter, and all internal angles of $\Fc$ are equal to $\frac{\pi}2$.

For each fundamental polygon $\Fc$ with sides along geodesics $P_iQ_{i+1}$, the \emph{Bowen--Series boundary map} $f_\P : \Sb \to \Sb$ is defined by
\< \label{map definition} f_\P(x) = T_i x \qquad \text{if }x\in[P_i,P_{i+1}). \>

The map admits a unique smooth ergodic invariant measure $\mu_\P$ (see~\cite[Theorem 1.2]{BS79}). Adler and Flatto~\cite{AF91} gave a thorough  analysis of these maps, their two-dimensional natural extensions, and applications to the symbolic coding of the geodesic flow on $\G\backslash\D$.
They also describe the measure $\mu_\P$ as a two-step projection of the invariant Liouville measure for the geodesic flow.

\smallskip
We can now state our first main result:
\begin{thm} \label{thm main formula}
	The entropy of the boundary map with respect to its smooth invariant measure is given by
	\< \label{main formula}
    	h_{\mu_\P}(f_\P) 
    	= \frac{\pi^2(4g-4)}{\text{\small$\mathrm{Perimeter}$}(\Fc)}
    	= \pi \cdot  \frac{\mathrm{Area}(\Fc)}{\text{\small$\mathrm{Perimeter}$}(\Fc)}.
    \> 
\end{thm}

\providecommand\Greg{{\G_{\!\mathrm{reg}}}}
\providecommand\Freg{{\Fc_{\!\mathrm{reg}}}}
Let $S=\G\backslash \D$ be any compact surface of genus $g\ge 2$ and $S_0=\Greg\backslash \D$ be a special genus $g$ surface that admits a regular $(8g-4)$-sided fundamental region $\Freg$.
By the Fenchel--Nielsen Theorem there exists an ori\-en\-ta\-tion-pre\-serv\-ing hom\-e\-o\-mor\-phism~$h$ from $\bar\D$ onto $\bar\D$ such that $\G=h\circ \Greg \circ h^{-1}$ and the sides of the fundamental polygon $\Fc$ for $\G$ belong to geodesics $P'_iQ'_{i+1}$, where $P'_i=h(P_i),Q'_{i+1}=h(Q_{i+1})$ and $P_iQ_{i+1}$ are the extensions of the sides of $\Freg$.
The map $h\big|_{\Sb}$ is a homeomorphism of $\Sb$ preserving the order of the points $\{P_i\}\cup\{Q_i\}$.

The Teichm\"uller space~$\Tc(S)$ of a surface $S$ can be thought of as any of the following:
\begin{enumerate}[\quad(i)]
	\item \label{item Teich Rss} the space of Riemann surface structures on $S$ modulo conformal maps isotopic to the identity~\cite[Section~1]{FM13};
	\item \label{item Teich mFg} the space of marked Fuchsian groups $\G$ such that $\pi_1 (S) \stackrel{\sim}{\to} \G$ and $S$ is ori\-en\-ta\-tion-pre\-serv\-ing homeomorphic to $\G\backslash\D$~\cite[Definition~2.1.1]{Funar};
	\item \label{item Teich mhp} the space of all marked canonical hyperbolic $(8g-4)$-gons in the unit disk~$\D$ such that side~$i$ and side~$\sigma(i)$ have equal length and the internal angles at vertices~$i$ and~$\sigma(i)\!+\!1$ sum to~$\pi$, up to an isometry of $\D$. (The topology on the space of polygons is as follows: $\Pc_k \to \Pc$ if and only if the lengths of all sides converge and the measures of all angles converge.)
\end{enumerate}
The space~$\Tc(S)$ is ho\-me\-o\-morph\-ic to~$\R^{6g-6}$. A standard way to parametrize $\Tc(S)$ is through Fenchel--Nielsen coordinates (see classical manuscript recently published in~\cite{FN}). The surface $S$ can be decomposed along $3g-3$ simple closed curves into $2g-2$ pairs of pants (shown for $g=3$ in the bottom of Figure~\ref{fig closed geodesics} on page~\pageref{fig closed geodesics}). For any $S'\in\Tc(S)$, these curves are canonically represented by geodesics, whose lengths determine each pair of pants up to isometry. To recover $S'$ we need in addition twist parameters when gluing pants together. Thus altogether $\Tc(S)$ is parametrized by~$\R_+^{3g-3}\times \R^{3g-3} $ (the first group of parameters are called the {\em lengths} and the second the {\em twists}), and $\dim \Tc(S)=6g-6$.

The construction (\ref{item Teich mhp}) of $\Tc(S)$ by varying ``marked'' fundamental  polygons is less common than the others. Following the earlier work~\cite{ZC80,T97}, Schmutz Schaller~\cite{S99} considers canonical $4g$-gons, but the canonical $(8g-4)$-gons may be considered as well. The following is a heuristic argument for the derivation of the dimension of $\Tc(S)$ using the $(8g-4)$-gon: the lengths of the identified pairs of sides are given by $4g-2$ real parameters; $2g-1$ real parameters represent the angles since four angles at each vertex are determined by one real parameter. The dimension of the space of isometries of $\D$ is $3$, so we remain with $(4g-2)+(2g-1)-3=6g-6$ parameters.

\medskip
A few years ago, Anatole Katok suggested a new area of research---or, at the very least, a new viewpoint---called the ``flexibility program,'' which can be broadly formulated as follows: under properly understood general restrictions, within a fixed class of smooth dynamical systems some dynamical invariants take arbitrary values. Taking this point of view, it is natural to ask how the     mea\-sure-theo\-re\-tic entropy $h_{\mu_\P}(f_\P)$ changes in $\Tc(S)$. Our second main result addresses this question:
\begin{thm}[Maximum and flexibility of entropy] \label{thm max and flexibility}~
    \begin{enumerate}[(i)]	
    	\item \label{item max} Among all surfaces in $\Tc(S)$, the maximum value  of the entropy $h_{\mu_\P}(f_\P)$ is achieved on the surface for which $\Fc$ is regular and is equal to 
    	\<\label{max H}
    	H(g) := h_{\mu_\P\reg}(f_\P\reg)=\frac{\pi^2(4g - 4)}{(8g-4)\cosh^{-1}(1+2\cos\tfrac\pi{4g-2})}.
    	\>
    	\item \label{item flexibility} For any value $h \in (0,H(g)]$ there exists $\Fc \in \Tc(S)$ such that $h_{\mu_\P}(f_\P) = h$.
    \end{enumerate}
\end{thm}

\medskip
The paper is organized as follows. In \Cref{sec formula proof} we prove \Cref{thm main formula}. The natural extension $F_\P$ of $f_\P$ and the ``geometric map'' $F\geo$ from~\cite{AK19} are used in the proof. In \Cref{sec flexibility proof} we prove \Cref{thm max and flexibility} by invoking the Isoareal Inequality and using Fenchel--Nielsen coordinates in the Teichm\"uller space related to a fundamental ($8g-4$)-gon. In \Cref{sec h top} we compare the topological entropy of the boundary map $f_{\bar P}$ to the measure-theoretic entropy and show that the smooth invariant measure $\mu_\P$ is not a measure of maximal entropy. In~\Cref{sec Maskit} we provide some computational tools for genus $2$.

\subsection*{Acknowledgements} The second author was partially supported by NSF grant DMS 1602409. The third author was partially supported by a Simons Foundation Collaboration Grant. The authors would like to thank the Institute of Mathematics of the Polish Academy of Sciences (IMPAN) for its hospitality. We would also like to thank the anonymous referee for careful reading and, in particular, for pointing out the correct constant relating $\d m$ and $\d\omega$ in \Cref{sec formula proof}.

\section{Proof of Theorem~\ref{thm main formula}} \label{sec formula proof}

The space of oriented geodesics on $\D$ is modeled as $\Sb \times \Sb \setminus \Delta$, where $\Delta$ is the diagonal $\setbuilder{(w,w)}{w\in\Sb}$.
The smooth measure
\[ \d\nu = \dfrac{\abs{\d u}\abs{\d w}}{\abs{u-w}^2} \]
on $\Sb\times\Sb\setminus\Delta$ was most probably first considered by E.~Hopf~\cite{H36} as he introduced the measure $\d m=\d\nu \d s$ on $T^1(\D)$ to study ergodic properties of the geodesic flow. The measure~$\d\nu$ was later used by Sullivan~\cite{Sul79}, Bonahon~\cite{B88}, Adler--Flatto~\cite{AF91}, and the current authors~\cite{KU17,AK19}. The measure $\d m$ is often more convenient for studying the geodesic flow than the Liouville volume 
\[ \d\omega = \frac{4\,\d x\,\d y\,\d\theta}{(1-x^2-y^2)^2}, \]
which comes from the hyperbolic measure on $\D$. Both measures~$\d\nu$ and~$\d m$ are preserved by M\"obius transformations, and $\d\omega = \tfrac12\d m$ (see~\cite[Appendix~A2]{B88}).\footnote{\,The constant relating $d\omega$ and $\d m$ was given incorrectly as $1/4$ in~\cite[page 250]{AF91}. Following that, $1/4$ was used in~\cite[Proposition~10.1]{AK19}.}

Adler and Flatto~\cite{AF91} introduced the ``rectilinear map'' defined by
\< F_\P(u,w) = (T_iu,T_iw) \qquad\text{if }w \in [P_i,P_{i+1}) \>
and showed the existence of an invariant domain $\Omega_\P \subset \Sb\times \Sb$ such that $F_\P$ re\-stric\-t\-ed to~$\Omega_\P$ is a two-dimensional geometric realization of the natural extension map of~$f_\P$. (In~\cite{KU17}, the authors showed that $\Omega_\P$ is also the global attractor of $F_\P:\Sb\times\Sb\setminus\Delta \to \Sb\times\Sb\setminus\Delta$.) The set $\Omega_\P$ is bounded away from the diagonal $\Delta$ and has a finite rectangular structure.
Thus $F_\P$ preserves the smooth probability measure
\[ \label{dnuP} \d\nu_\P := \frac{\d\nu}{\int_{\Omega_\P}\d\nu}. \]
The boundary map $f_\P$ is a factor of $F_\P$ (projecting on the second coordinate), so one can obtain its smooth invariant probability measure $\mu_\P$ as a projection.

\medskip
The geodesic flow on $S$ can be realized as a special flow over a cross-section that is parametrized by $\Omega_\P$, and the first return map to this cross-section acts exactly as $F_\P : \Omega_\P \to \Omega_\P$.
Using this realization along with Abramov's formula and the Ambrose--Kakutani theorem, we have from~\cite[Proposition~10.1]{AK19} (with a corrected constant) that 
\[ h_{\mu_\P}(f_\P) = h_{\nu_\P}(F_\P) = \frac{\pi^2(4g-4)}{\int_{\Omega_\P}\d\nu}, \]
and since $\mathrm{Area}(\Fc)=2\pi(2g-2)$ by the Gauss--Bonnet formula, we have 
\< \label{h from AK} h_{\mu_\P}(f_\P) = \pi \cdot \frac{\mathrm{Area}(\Fc)}{\int_{\Omega_\P}\d\nu}. \>

To prove \Cref{thm main formula}, it remains only to show that $\int_{\Omega_\P}\d\nu$ is equal to the (hyperbolic) perimeter of $\Fc$.
For that, we use another map, also introduced by Adler--Flatto in~\cite{AF91}, called the ``curvilinear map'' (or ``geometric map'' in~\cite{AK19}). Denoting by $uw$ the geodesic from $u$ to $w$, the map is defined on the set
\[ \Omega\geo:=\setbuilder{ (u,w) }{ uw \text{ intersects }\Fc} \;\;\subset\;\; \Sb\times\Sb\setminus\Delta \]
and is given by
\[ F\geo(u,w) = (T_iu,T_iw) \quad\text{if $uw$ exits $\Fc$ through side $i$.} \]
There is a key correspondence between $\Omega\geo$ and $\Omega_\P$ (see \Cref{fig bulges and corners}):
\begin{prop}[{\cite[Theorem~5.1]{AF91}}] \label{thm bijection}
	The map $\Phi : \Omega\geo \to \Omega_\P$ given by
    \[ \Phi = \left\{ \begin{array}{ll}
        \Id & \text{on } \Omega\geo \cap \Omega_\P \\
        T_{\sigma(i)-1}T_i & \text{on } \Bc_i,
    \end{array} \right. \]
    where $\Bc_i = \setbuilder{(u,w) \in \Omega\geo \setminus \Omega_\P}{w \in [P_i,P_{i+1}]}$, is bijective.
\end{prop}
\begin{figure}[htb]
	\includegraphics[width=0.46\textwidth]{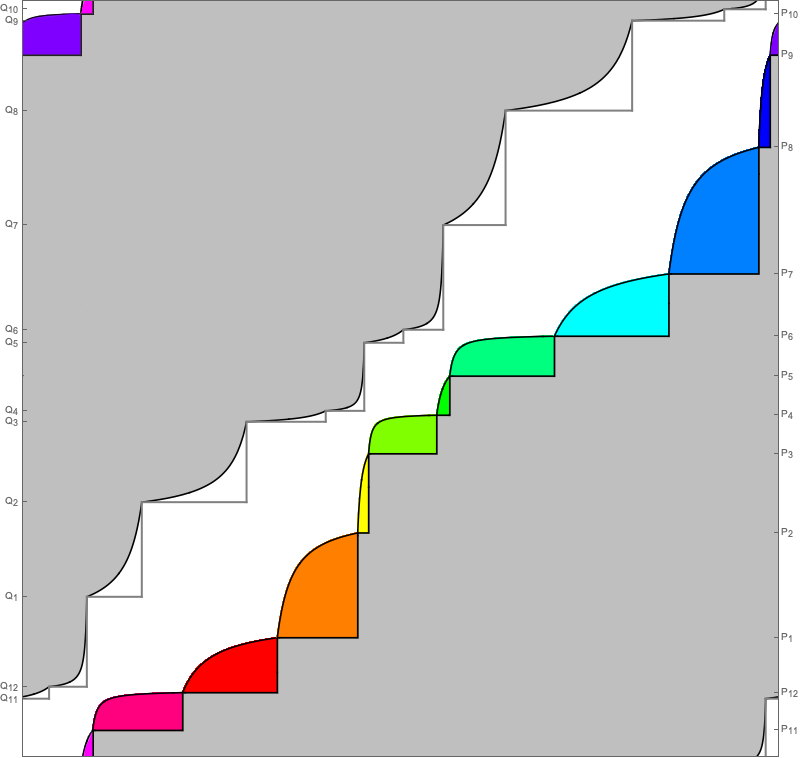}
	\raisebox{2.5cm}{$\stackrel{\Phi}{\to}$}
	\includegraphics[width=0.46\textwidth]{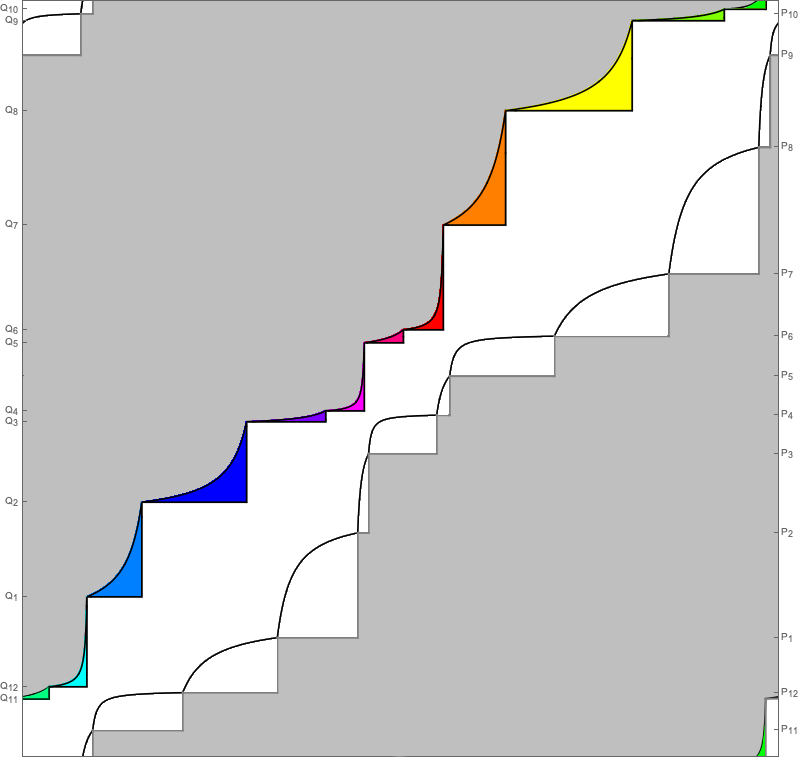} 
    \caption{Bulges $\Bc_i$ of $\Omega\geo$ (left) are mapped to corners of $\Omega_\P$ (right)\captionperiod}
    \label{fig bulges and corners}
\end{figure}
	
Since $\Phi$ acts by fractional linear transformations, which preserve the measure $\nu$, we have that
\<
    \label{PG} \int_{\Omega_\P}\d\nu = \int_{\Omega\geo}\d\nu.
\>

Having proved \eqref{PG}, we now want to show that $\int_{\Omega\geo}\d\nu$ is equal to the (hyperbolic) perimeter of~$\Fc$.
	
\begin{lem}[{(Bonahon)}] \label{lem length from integral}
	For any oriented geodesic segment $s$ on $\D$, 
	\[ \int_{\Psi^+\!(s)}\d\nu = \mathrm{length}(s), \]
	where $\Psi^+\!(s)$ is the set of oriented geodesics intersecting $s$ with the oriented angle at the intersection between $0$ and $\pi$.
\end{lem}

The proof involves expressing $\d\nu = |\d u||\d w|/|u-w|^2$ in a coordinate system~$(x,\theta)$ based on movement along geodesics, namely,
\[ \d\nu = \tfrac12 \sin(\theta) \,\d\theta\,\d x, \]
where $x$ is the distance along the segment $s$ from the point of intersection of $s$ with $uw$, and $\theta$ is the angle that $uw$ makes with $s$. (This measure is sometimes called ``geodesic current.'')  See~\cite[Appendix~A3]{B88} for details.\footnote{\,Thank you to Alena Erchenko for providing this reference.}

\medskip
Recall that the domain $\Omega\geo$ of the geometric map $F\geo$ consists of all $(u,w)$ for which $uw$ intersects $\Fc$. This can be decomposed as $\Omega\geo = \bigcup_{i=1}^{8g-4} \Gc_i$, where  \[ \Gc_i = \setbuilder{ (u,w) }{ uw\text{ exits $\Fc$ through side }i } = \Psi_+(\text{side }i) \]
(these ``strips'' are shown in~\cite[Figure~3]{AK19}). Thus from \Cref{lem length from integral} we immediately get 
\[ \label{Gperim} \int_{\Omega\geo}\d\nu = \sum_{i=1}^{8g-4} \int_{\Gc_i}\d\nu = \sum_{i=1}^{8g-4} \mathrm{length}(\text{side }i) = \mathrm{Perimeter}(\Fc). \]
Combining this with \eqref{PG}, one can replace $\int_{\Omega_\P}\d\nu$ by the perimeter of $\Fc$ in the denominator of \eqref{h from AK}; this completes the proof of \Cref{thm main formula}.

\begin{remark}
In~\cite{KU17}, the authors introduced and investigated dynamical properties of boundary map $f_{\bar A}$ defined for an arbitrary multi-parameter $\bar A=\{A_1,A_2,\dots,A_{8g-4}\}$ with all $A_i\in (P_i,Q_i)$ satisfying the so-called ``short cycle property''
$f_{\bar A}(T_iA_i)=f_{\bar A}(T_{i-1}A_i)$.
It was proved in \cite{AK19} that $\nu(\Omega_{\bar A})=\nu(\Omega_{\bar P})$ and $F_{\bar A}$ and $F_{\bar P}$ are measure-theoretically isomorphic, which implies that $h_{\nu_{\bar A}}(F_{\bar A})=h_{\nu_\P}(F_\P)$. This allows us to conclude that $\mu_{\bar A}=\mu_{\bar P}$ and to prove the same formula~\eqref{main formula} for the entropy of $f_{\bar A}$:
\begin{equation*} 
    	h_{\mu_{\bar A}}(f_{\bar A}) 
    	= \frac{\pi^2(4g-4)}{\text{\small$\mathrm{Perimeter}$}(\Fc)}
    	= \pi \cdot  \frac{\mathrm{Area}(\Fc)}{\text{\small$\mathrm{Perimeter}$}(\Fc)}=h_{\mu_\P}(f_\P).
\end{equation*}
In other words, the entropy remains unchanged for all boundary maps $f_{\bar A}$ defined using a partition $\bar A$ satisfying the short cycle property.
\end{remark}

\section{Proof of Theorem~\ref{thm max and flexibility}} \label{sec flexibility proof}

To prove that \Cref{thm max and flexibility}(\ref{item max}) follows from \Cref{thm main formula}, we only need to show that for each genus $g$ the perimeter of $\Fc$ in $\Tc(S)$ is minimized on the regular polygon.
\begin{thm}[Isoareal Inequality] \label{thm isoareal}
	Among all hyperbolic polygons with a given area and number of sides, the regular polygon has the smallest perimeter.
\end{thm}

\begin{proof}
For a hyperbolic $n$-gon $\Pc_n$, the inequality 
\[ \mathrm{Perimeter}(\Pc_n)^2 \ge 4\,d_n\,\mathrm{Area}(\Pc_n), \qquad d_n = n \tan\big(\tfrac{\mathrm{Area}(\Pc_n)}{2n}\big), \]
is given in~\cite[Theorem 1.2(a)]{Ku}, which also states that equality is achieved on a regular polygon. Both isoperimetric and isoareal inequalities follow: $\mathrm{Area}(\Pc_n)$ and $n$ are constant, so the right-hand side $4 d_n \mathrm{Area}(\Pc_n)$ is constant and thus the perimeter of $\Fc$ is minimal when $\Fc$ is a regular polygon.
\end{proof}

In our setting, $\Fc = \Pc_n$ with $n=8g-4$ and $\mathrm{Area}(\Fc) = 2\pi(2g-2)$ is constant in $\Tc(S)$, so by \Cref{thm isoareal} the perimeter is minimized when $\Fc$ is regular.
The expression for the maximum value $H(g)$ in \eqref{max H} comes directly from \eqref{main formula}, with \[ \cosh^{-1}\big(1+2\cos\tfrac\pi{4g-2}\big) \] being the length of a single side of the regular $(8g-4)$-gon. This completes the~proof of \Cref{thm max and flexibility}(\ref{item max}).

To prove \Cref{thm max and flexibility}(ii), we recall that
Fen\-chel--Niel\-sen coordinates use a decomposition of $S$ into $2g-2$ pairs of pants by $3g-3$ non-intersecting closed geodesics whose lengths can be manipulated independently (these lengths form $3g-3$ of the $6g-6$ coordinates).
We take one of these geodesics to also be a geodesic from the chain described in \Cref{sec intro} that corresponds to one entire side of $\Fc$ (this shared geodesic is on the far right in both parts of \Cref{fig closed geodesics}). Since the length of this side (one of the Fenchel--Nielsen coordinates) can be made arbitrarily large, the perimeter of $\Fc$ can also be made arbitrarily large, which by~\eqref{main formula} means that $h_{\mu_\P}(f_\P)$ can be made arbitrarily small.

\begin{figure}[htb]
    \includegraphics[width=0.875\textwidth]{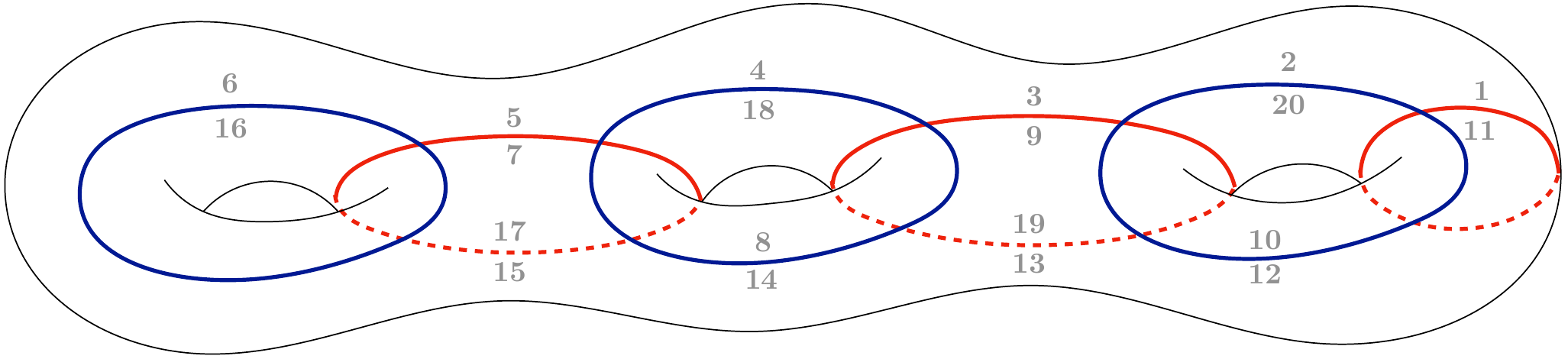} \\[0.25em]
    \includegraphics[width=0.875\textwidth]{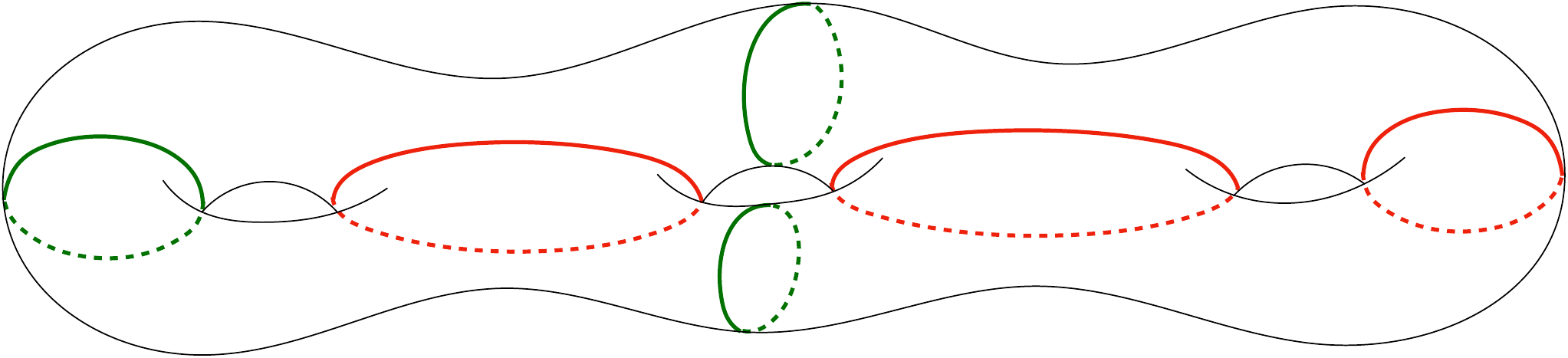}
    \caption{Chain of $2g$ geodesics on $S$ forming the sides of $\Fc$ (top) and decomposition of $S$ into $2g-2$ pairs of pants by $3g-3$ non-intersecting geodesics (bottom) for $g=3$\captionperiod}
    \label{fig closed geodesics}
\end{figure}

Using the continuity of the Fenchel--Nielsen coordinates, if $\G\to \G'$  in $\Tc(S)$, then, by the Fenchel--Nielsen Theorem, $\G=h\circ\G'\circ h^{-1}$ for some orientation preserving homeomorphism $h:\bar\D\to\bar\D$, and
$h|_\Sb\to\Id$ as circle homeomorphisms, i.e., $d(h(x),x)\to 0$ for all $x\in\Sb$. Therefore, for the endpoints of the geodesics $P_iQ_{i+1}$ containing the sides of the fundamental polygon $\Fc$ and the geodesics $P'_iQ'_{i+1}$ containing the sides of the fundamental polygon $\Fc'$, we have $P_i\to P'_i$ and $Q_{i+1}\to Q'_{i+1}$. It follows that the vertices of $\Fc$ will tend to the vertices of $\Fc'$, and hence $\mathrm{Perimeter}(\Fc)\to\mathrm{Perimeter}(\Fc')$, i.e., the perimeter of $\Fc$ varies continuously 
within the Teichm\"uller space~$\Tc(S)$. 
From \Cref{thm main formula} we conclude the continuity of the entropy  $h_{\mu_\P}(f_\P)$ 
within~$\Tc(S)$. By the Intermediate Value Theorem, $h_{\mu_\P}(f_\P)$ must take on all values between $0$ and its maximum.

\medskip
For genus $2$, the techniques of Maskit (see \Cref{sec Maskit}) allow us to accurately draw the fundamental polygon $\Fc$ for any values of the Fenchel--Nielsen coordinates. \Cref{fig entropy vs side length} shows how the entropy changes as the single Fenchel--Nielsen coordinate representing the length of the bottom side of $\Fc$ is varied.

\begin{figure}[hbt]
    \includegraphics[width=0.85\textwidth]{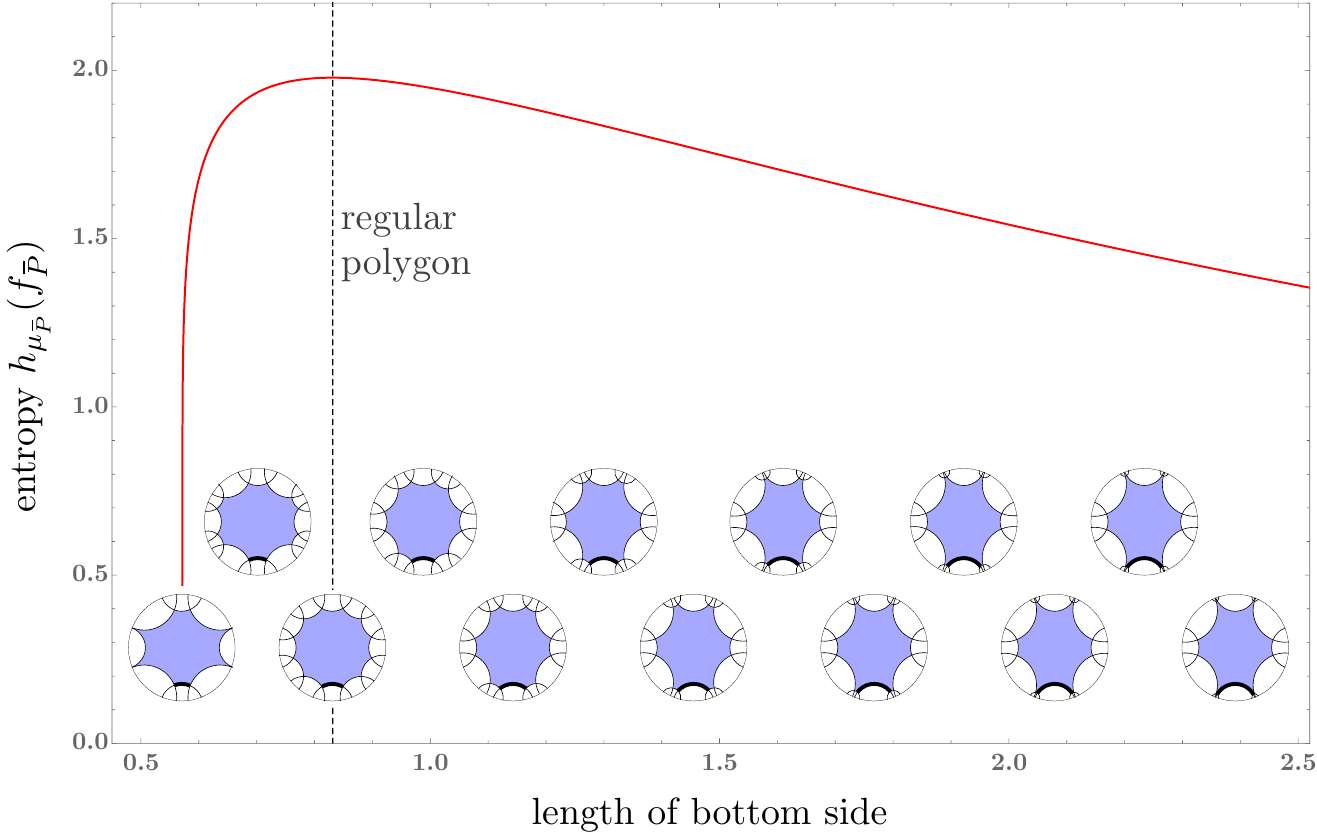}
    \caption{Entropy as a function of a single Fenchel--Nielsen coordinate for $g=2$\captionperiod}
    \label{fig entropy vs side length}
\end{figure}

\section{Topological entropy} \label{sec h top}

The notion of topological entropy was originally introduced for continuous maps acting on compact metric spaces. As explained in~\cite{MZ}, Bowen's definition can also be applied to piecewise continuous, piecewise monotone maps on an interval. The theory naturally extends to maps of the circle, where monotonicity is understood to mean local monotonicity.

The map $f_\P$ is Markov with respect to the partition $\{I_1,\dots,I_{16g-8}\}$ given by
\begin{equation} \label{MarkovPartition}
I_{2i-1} := [P_i, Q_i), \qquad I_{2i} := [Q_{i}, P_{i+1}), \qquad i=1,\dots,8g-4
\end{equation}
(see~\cite[Lemma 2.5]{BS79} or~\cite[Theorem 6.1]{AF91}). The associated transition matrix $M$ has entries 
\[ m_{ij} = \left\{\begin{array}{ll} 1 &\text{if } f_\P(I_i) \supset I_j \\ 0 &\text{otherwise,} \end{array}\right. \]
and the topological entropy of $f_\P$ is 
\[ h_\mathrm{top}(f_\P) = \log\abs{\lambda_\mathrm{max}}, \]
where $\abs{\lambda_\mathrm{max}}$ is the spectral radius (that is, the eigenvalue with largest absolute value) of $M$ (see, e.g., \cite[Proposition~3.2.5]{KH}). 

In some situations the Lebesgue measure $\mu$ will satisfy $h_\mu(f) = h_\mathrm{top}(f)$, but the boundary map $f_\P: \Sb \to \Sb$ provides an example where the smooth invariant measure~$\mu_\P$ is not a measure of maximum entropy (since $f_\P$ is Markov, the measure of maximal entropy is the Parry measure).

It is a direct calculation that $\lambda = 4g-3 + \sqrt{(4g-3)^2-1}$ is an eigenvalue of $M$ with corresponding eigenvector
\[ v = (1,\; \lambda\!-\!1,\; 1,\; \lambda\!-\!1,\; \dots,\; 1,\; \lambda\!-\!1). \]
This shows that the topological entropy  satisfies
\< \label{htop inequality} h_\mathrm{top}(f_\P) \ge \log\big( 4g - 3 + \sqrt{(4g-3)^2 - 1}\big), \> which implies \Cref{thm less} below. 

We should point out that as we move in the Teichm\"uller space~$\Tc(S)$, by the Fenchel-Nielsen Theorem mentioned in the Introduction, the partition (\ref{MarkovPartition}) of $\Sb$ into $16g-8$ intervals  remains Markov with the same transition matrix $M$, therefore $h_\mathrm{top}(f_{\bar P})$ does not change.

\begin{cor} \label{thm less}
    The mea\-sure-theo\-re\-tic entropy of $f_\P$ with respect to its smooth invariant measure~$\mu_\P$ is strictly less than the topological entropy of $f_\P$. 
\end{cor}

\begin{proof}
From \eqref{max H}, we have that $H(g)=h_{\mu_\P\reg}(f_\P\reg)$, computed in \Cref{thm max and flexibility}(i), is an increasing function of~$g$, and we can calculate
\[
    \lim_{g\to\infty}H(g)  
    = \lim_{g\to\infty} \frac{\pi^2(4g-4)}{(8g\!-\!4)\cosh^{-1}(1\!+\!2\cos\tfrac\pi{4g-2})}
    = \frac{\pi^2}{2\cosh^{-1}(3)}
.\]
Since $H(g)$ is increasing, its value for any $g$ is less than or equal to this limit. The function
\[ \log\big( 4g - 3 + \sqrt{(4g-3)^2 - 1}\big) \]
is also increasing, so, by \eqref{htop inequality}, for any $g \ge 3$ we have
\[ 
     h_\mathrm{top}(f_\P)
     \ge \log(9 + 4\sqrt{5})\approx 2.8872
\]
and, therefore,
\[
    h_{\mu_\P}(f_\P) \;\le \; \frac{\pi^2}{2\cosh^{-1}(3)} \;<\;
    2.8 \;<\; 
    \log(9 + 4\sqrt{5}) \;\le\;
    h_\mathrm{top}(f_\P).
\]
The case $g=2$ is checked separately:
\[ h_{\mu_\P}(f_\P)=\frac{\pi^2}{3\cosh^{-1}(1+\sqrt{3})}\approx 1.9784 
,\quad h_\mathrm{top}(f_\P) \ge \log(5+2\sqrt6)\approx 2.2924. \]
This completes the proof of the \namecref{thm less}.
\end{proof}

\begin{remark} In an upcoming paper~\cite{AKUpre-print} we prove that $4g-3 + \sqrt{(4g-3)^2-1}$ is the maximal eigenvalue of $M$, thus making \eqref{htop inequality} an equality and obtaining the exact formula for $h_\mathrm{top}(f_\P)$. For \Cref{thm less}, however, the inequality is sufficient.
\end{remark}

\appendix
\section{Appendix. Computational tools for genus 2} \label{sec Maskit}

The polygon in \Cref{fig irregular sides} and the details of \Cref{fig bulges and corners} were produced using the generators of $\G$ in terms of the Fenchel--Nielsen coordinates $(\alpha, \beta, \gamma, \sigma, \tau, \rho)$ for $g=2$ introduced by Maskit~\cite{M99}. In case they will be useful for others, we provide below the relevant information for doing numerical experiments in $\Tc(S)$ for genus $2$.

\providecommand\Matrix[1]{\begin{pmatrix}#1\end{pmatrix}}
Maskit uses the six parameters above along with
\begin{align*}
	\mu &= \cosh^{-1}\!\big( \coth\beta\cosh\sigma\cosh\tau + \sinh\sigma\sinh\tau \big) \\[0.6em]
	\delta &= \coth^{-1}\!\Big(\frac{\cosh\gamma\cosh\mu-\coth\alpha\sinh\gamma\sinh\mu-\sinh\sigma\sinh\rho}{\cosh\sigma\cosh\rho}\Big)
\end{align*}
to define matrices $\tilde A$, $\tilde B$, $\tilde C$, $\tilde D$ acting on the half-plane.
Setting $A = \frac12 (\begin{smallmatrix}i & 1 \\ 1 & i\end{smallmatrix}) \tilde A (\begin{smallmatrix}-i & 1 \\ 1 & -i\end{smallmatrix})$, etc., we get the following matrices acting on the disk:
\begin{align*}
    A &= \frac{\sinh\alpha}{\sinh\mu} \Matrix{
   		\coth \alpha \sinh \mu + i & -i \cosh \mu \\
		i \cosh \mu & \coth \alpha \sinh \mu - i
	} \\[0.6em]
	B &= \frac{\sinh\beta}{\cosh\tau} \Matrix{
   		\cosh \tau \coth \beta + i \sinh \sigma & \cosh \sigma + i \sinh \tau \\
		\cosh \sigma - i \sinh \tau & \cosh \tau \coth \beta - i \sinh \sigma
	} \\[0.6em]
	C &= \Matrix{
   		\cosh \gamma & i \sinh \gamma \\
		-i \sinh \gamma & \cosh \gamma
	} \\[0.6em]
	D &= \frac{\sinh\delta}{\cosh\rho} \Matrix{
   		\cosh \rho \coth \delta - i \sinh (\gamma +\sigma ) & -\cosh (\gamma +\sigma ) - i \sinh \rho \\
		-\cosh (\gamma +\sigma ) + i \sinh \rho & \cosh \rho \coth \delta + i \sinh (\gamma +\sigma )
	}.
\end{align*}

Let $S_i$ be the transformation for which $P_i$ is the repelling fixed point and $Q_{i+1}$ is the attracting fixed point. That is, the oriented axis of $S_i$ contains side $i$. 
We have
\begin{align*}
	S_1 &= C^{-1}D^{-1}C &	
	S_2 &= AC &
	S_3 &= ABDA^{-1} &
	S_4 &= A^{-1} \\*
	S_5 &= D^{-1}B^{-1} & 
	S_6 &= CA &
	S_7 &= D &
	S_8 &= DA^{-1}C^{-1}D^{-1} \\*
	S_9 &= B^{-1}D^{-1} &
	S_{10} &= B^{-1}AB &
	S_{11} &= C^{-1}DCB &
	S_{12} &= C^{-1}B^{-1}A^{-1}B.
\end{align*}
The side-pairing transformations
~are
\begin{align*}
	T_1 &= C & 
	T_2 &= C^{-1}DC & 
	T_3 &= A^{-1} &
	T_4 &= B^{-1} \\*
	T_5 &= A &   
	T_6 &= D &
	T_7 &= C^{-1} &
	T_8 &= D^{-1} \\*
	T_9 &= B^{-1}AB &
	T_{10} &= B &
	T_{11} &= B^{-1}A^{-1}B &
	T_{12} &= C^{-1}D^{-1}C,
\end{align*}
and the defining relation 
\[ ABDA^{-1}C^{-1}D^{-1}CB^{-1} = \mathrm{Id} \]
from~\cite{M99} is equivalent to 
\cite[Equation~1.5]{KU17} with $g=2$. 
The regular $12$-gon corresponds to values \[ \alpha = \tfrac12 \operatorname{arccosh}(1+\sqrt3), \quad \beta = \gamma = 2\alpha, \quad \sigma = \tau = \rho = 0 \] for Maskit's Fenchel--Nielsen coordinates.

\end{document}